\documentclass[11pt]{article}
\usepackage{amsmath, amssymb, amscd, amsthm, amsfonts}
\usepackage{graphicx}
\usepackage{hyperref}
\usepackage{amssymb}
\usepackage{amsmath}
\usepackage{amsthm}
\usepackage{amsfonts}
\usepackage{a4wide}
\usepackage{longtable}
\usepackage{multirow}
\usepackage{systeme}
\usepackage[latin1]{inputenc}
\usepackage[english]{babel}
\usepackage[T1]{fontenc}
\usepackage{amssymb,amsmath,amsthm,amstext,amsfonts,latexsym}

\usepackage{pgf,tikz}
\usepackage{mathrsfs}
\usetikzlibrary{arrows}

\newtheorem{lemma}{Lemma}[section]
\newtheorem{proposition}{Proposition}[section]
\newtheorem{theorem}{Theorem}[section]
\newtheorem{corollary}{Corollary}[section]

\newtheorem{remark}{Remark}[section]
\theoremstyle{definition}

\newtheorem{Properties}{Properties}[section]
\theoremstyle{remark}

\usepackage[all]{xy}
\title{\huge{Generators of the $5$-class group of fields of degree 20}}
\author{ ABDELMALEK AZIZI\and  FOUAD ELMOUHIB \and MOHAMED TALBI}
\date{}

\begin{document}

\maketitle

\begin{center}
{\sc Abdelmalek AZIZI }\\
{\footnotesize Department of Mathematics and Computer Sciences,\\
Mohammed First University, Oujda, Morocco,\\
abdelmalekazizi@yahoo.fr}\\
\vspace{0.7cm}
{\sc Fouad ELMOUHIB }\\
{\footnotesize Department of Mathematics and Computer Sciences,\\
Mohammed First University, Oujda, Morocco,\\
Correspondence: fouad.cd@gmail.com}\\
\vspace{0.7cm}
{\sc Mohamed TALBI }\\
{\footnotesize Regional Center of Professions of Education and Training, Oujda, Morocco,\\
ksirat1971@gmail.com}\\

\end{center}

\begin{abstract}
Let $\Gamma  =  \mathbb{Q}(\sqrt[5]{n})$ be a pure quintic field, where $n$ is a natural number $5^{th}$ power-free, $k_0 = \mathbb{Q}(\zeta_5)$ be the cyclotomic field containing a primitive $5^{th}$ root of unity $\zeta_5$ and $k = \Gamma(\zeta_5)$ be the normal closure of $\Gamma$. Let $C_{k,5}$ be the $5$-component of the class group of k. The purpose of this paper is to write down the generators of $C_{k,5}$, whenever it is of type $(5,5)$ and the rank of the group of ambiguous classes under the action of $\mathrm{Gal}(k/k_0)$ is $1$.
\end{abstract}

\section{Introduction}\label{intro}
Let $\Gamma  =  \mathbb{Q}(\sqrt[5]{n})$ be a pure quintic field, where $n$ is a natural number $5^{th}$ power-free, $k_0 = \mathbb{Q}(\zeta_5)$ the cyclotomic field containing a primitve $5^{th}$ root  of unity, then $k = \Gamma(\zeta_5)$ is the normal closure of $\Gamma$. Let $C_{k,5}$ be the $5$-class group of $k$ and $C_{k,5}^{(\sigma)}$ the subgroup of ambiguous ideal classes under the action of $\mathrm{Gal}(k/k_0)\, =\,\langle \sigma\rangle$. In \cite{FOU} we showed that there are six possible shapes of the radicand $n$ for which it is possible that $C_{k,5}$ is of type $(5,5)$ and $C_{k,5}^{(\sigma)} \simeq \mathbb{Z}/5\mathbb{Z}$. Based on some exhaustive numerical
calculus realized by the system PARI/GP \cite{PRI}, we conjectured that, among the six possible forms of $n$, only for three of them one can have $C_{k,5}$ of type $(5,5)$ and $C_{k,5}^{(\sigma)} \simeq \mathbb{Z}/5\mathbb{Z}$ as follows:\\

$(1)$ $n = 5^{e}p\,\not\equiv\,\pm1,\pm7\,(\mathrm{mod}\,25)$ such that $p\,\not\equiv\,-1\,(\mathrm{mod}\,25)$.\\

$(2)$ $n = p^{e_1}q\equiv\pm1,\pm7\,(\mathrm{mod}\,25)$ such that $p\,\not\equiv\,-1\,(\mathrm{mod}\,25)$, $q\,\not\equiv\,\pm7\,(\mathrm{mod}\,25)$.\\

$(3)$ $n = p^{e_1}\equiv\pm1,\pm7\,(\mathrm{mod}\,25)$ such that $p\equiv-1\,(\mathrm{mod}\,25)$.\\
Here $p$ and $q$ are primes such that $p\equiv-1\,(\mathrm{mod}\,5)$, $q\equiv\pm2\,(\mathrm{mod}\,5)$, $0 \leq e \leq 4$ and $1 \leq e_1 \leq 4$.\\
In this paper, we will prove that $C_{k,5}$ is of type $(5,5)$ if and only if $5$ divides exactly the $5$-class number of $\Gamma$ and $u = 5^3$, where $u$ is the index of the subgroup $E_0$ generated by the units of intermediate fields of the extension $k/\mathbb{Q}$ in the unit group of $k$. Next we determine generators of $C_{k,5}$, when $C_{k,5}$ is of type $(5,5)$ and rank $C_{k,5}^{(\sigma)} = 1$, for each of the three forms of $n$ given above.\\
Since $k$ is a Kummer extension of $k_0$, we begin with the decomposition laws in Kummer extensions, which is useful to give the prime factorization in the normal closure $k$. As the proof of our main results is established by the norm residue symbol, we recall the definition and the most important properties of this concept. Making use of this, we can determine generators of the $5$-class group $C_{k,5}$ of type $(5,5)$ if rank $C_{k,5}^{(\sigma)} = 1$. Our main result will be  underpinned by numerical examples, using the PARI/GP [\cite{PRI}] in section 7. In fact, we shall prove the following Main Theorem:
\begin{theorem}
\label{thp}
Let $k = \mathbb{Q}(\sqrt[5]{n},\zeta_5)$, where $n$ is a natural number $5^{th}$ power-free, be the normal closure of the pure quintic field $\Gamma = \mathbb{Q}(\sqrt[5]{n})$. Let $\mathrm{Gal}(k/\Gamma) = \langle\tau\rangle$. Let $p$,$q$ and $l$ be primes such that, $p\equiv-1\, (\mathrm{mod}\, 5)$, $q\equiv\pm2\, (\mathrm{mod}\, 5)$ and $l \neq p$, $l \neq q$. Assume that $C_{k,5}$ is of type $(5,5)$ and rank$(C_{k,5}^{(\sigma)}) = 1$, then we have:
\begin{itemize}
\item[$(1)$] If $n = 5^{e_1}p^{e_2}\,\not\equiv\,\pm1,\pm7\,(\mathrm{mod}\,25)$, with $e\in\{1,2,3,4\}$ and $p\,\not\equiv\,-1\, (\mathrm{mod}\, 25)$. The prime $p$ decomposes in $k$ as $p\mathcal{O}_k = \mathcal{P}^5_1\mathcal{P}^5_2$, where $\mathcal{P}_1$ and $\mathcal{P}_2$ are prime ideals of $k$. Let $\mathcal{L}$ be a prime ideal of $k$ above $l$. If $5$ and $l$ are not quintic residues modulo $p$, then the $5$-class group $C_{k,5}$ is generated by classes $[\mathcal{P}_1]$ and $[\mathcal{L}]^{1-\tau^2}$ and we have:
\begin{center}
$C_{k,5}$ = $\langle[\mathcal{P}_1]\rangle\times\langle[\mathcal{L}]^{1-\tau^2}\rangle$ = $\langle[\mathcal{P}_1],[\mathcal{L}]^{1-\tau^2}\rangle$ 
\end{center}

\item[$(2)$] If $n = p^{e_1}q^{e_2}\equiv\pm1,\pm7\,(\mathrm{mod}\,25)$, with $e\in\{1,2,3,4\}$ and $p\,\not\equiv\,-1\, (\mathrm{mod}\, 25)$, $q\,\not\equiv\,\pm7\, (\mathrm{mod}\, 25)$. The prime $p$ decomposes in $k$ as $p\mathcal{O}_k = \mathcal{P}^5_1\mathcal{P}^5_2$, where $\mathcal{P}_1$ and $\mathcal{P}_2$ are prime ideals of $k$. Let $\mathcal{L}$ be a prime ideal of $k$ above $l$. If $q$ and $l$ are not a quintic residues modulo $p$, then the $5$-class group $C_{k,5}$ is generated by classes $[\mathcal{P}_1]$ and $[\mathcal{L}]^{1-\tau^2}$ and we have:
\begin{center}
$C_{k,5}$ = $\langle[\mathcal{P}_1]\rangle\times\langle[\mathcal{L}]^{1-\tau^2}\rangle$ = $\langle[\mathcal{P}_1],[\mathcal{L}]^{1-\tau^2}\rangle$ 
\end{center}

\item[$(3)$] If $n = p^e\equiv\pm1,\pm7\,(\mathrm{mod}\,25)$, with $e\in\{1,2,3,4\}$ and $p\equiv-1\, (\mathrm{mod}\, 25)$. $5$ decomposes in $k$ as $5\mathcal{O}_k = \mathcal{B}^4_1\mathcal{B}^4_2\mathcal{B}^4_3\mathcal{B}^4_4\mathcal{B}^4_5$, where $\mathcal{B}_i$  are prime ideals of $k$. If $5$ is not a quintic residue modulo $p$, then the $5$-class group $C_{k,5}$ is generated by classes $[\mathcal{B}_i]$ and $[\mathcal{B}_j]$, $i\,\neq\,j\in\{1,2,3,4,5\}$ and we have:
\begin{center}
$C_{k,5}$ = $\langle[\mathcal{B}_i]\rangle\times\langle[\mathcal{B}_j]\rangle$ = $\langle[\mathcal{B}_i],[\mathcal{B}_j]\rangle$ 
\end{center}

\end{itemize}
\end{theorem}

\begin{center}
$\textbf{Notations. \ }$
\end{center}
Throughout this paper, we use the following notations:
\begin{itemize}
 
 \item The lower case letters $p$,$q$ and $l$ will denote prime numbers such that, $p\equiv-1\, (\mathrm{mod}\, 5)$, $q\equiv\pm2\, (\mathrm{mod}\, 5)$ and $l\,\neq\,p$, $l\,\neq\,q$ .
 
 \item $\Gamma = \mathbb{Q}(\sqrt[5]{n})$: a pure quintic field, where $n\neq1$ is a $5^{th}$ power-free natural number.
 
 \item $k_0 = \mathbb{Q}(\zeta_5)$: the cyclotomic field, where $\zeta_5 = e^{\frac{2i\pi}{5}}$ a primitive $5^{th}$ root of unity.
 
 \item $k = \mathbb{Q}(\sqrt[5]{n},\zeta_5)$: the normal closure of $\Gamma$, a quintic Kummer extension of $k_0$.
 
 \item $\langle\tau\rangle = \operatorname{Gal}(k/\Gamma)$ such that $\tau$ is identity on $\Gamma$, and sends $\zeta_5$ to its square. Hence $\tau$ has order 4.
 
 \item $\langle\sigma\rangle = \operatorname{Gal}(k/k_0)$ such that $\sigma$ is identity on $k_0$, and sends $\sqrt[5]{n}$ to $\zeta_5\sqrt[5]{n}$. Hence $\sigma$ has order 5.

\item $u$: the index of subgroup $E_0$ generated by the units of intermediate fields of the extension\\\\ $k/\mathbb{Q}$ in the unit group of $k$.

 \item $\lambda = 1-\zeta_5$ is prime element above $5$ in $k_0$.
 
 \item $q^{\ast} = 2,\,1$ or $0$, according to whether $\zeta_{5}$ and $1+\zeta_5$ are both, or only one of them, or none is norm or is norm of an element of $k^* = k\setminus\lbrace 0\rbrace$.
 
 \item For a number field $L$, denote by:
  \begin{itemize}
   \item $\mathcal{O}_{L}$: the ring of integers of $L$;
   \item $h_{L}$: the class number of $L$;
   \item $C_{L,5}$: the $5$-class group of $L$;
   \item $[\mathcal{I}]$ : the class of a fractional ideal $\mathcal{I}$ in the class group of L.
  \end{itemize}
  
\end{itemize}



\section{\Large Decomposition laws in Kummer extensions}
\label{deckumm}
Since the $5$-extensions of $k$ and $k_0$ are all Kummer extensions, we recall the decomposition laws of ideals in these extensions. Let $L$ be a number field which contains the $l^{th}$ roots of unity, where $l$ is prime, and $\theta$ be an element of $L$, such that $\theta\,\neq\,\mu^l$, for all $\mu \in L$, therefore $L(\sqrt[l]{\theta})$ is cyclic extension of degree $l$ over $L$. We note by $\zeta$ a $l^{th}$ primitive root of unity. 
\begin{proposition}.\\\label{prKummer}
$(1)$ We assume that a prime $\mathcal{P}$ of $L$ divides $\theta$ exactly to the power $\mathcal{P}^a$.
\begin{itemize}
\item[-] If $a=0$ and $\mathcal{P}$ does not divide $l$, then $\mathcal{P}$ splits completely in $L(\sqrt[l]{\theta})$ when the congruence\\ $\theta\equiv X^l\,(\mathrm{mod}\,\mathcal{P})$ has a solution in $L$.
\item[-] If $a=0$ and $\mathcal{P}$ does not divide $l$, then $\mathcal{P}$ is inert in $L(\sqrt[l]{\theta})$ when the congruence $\theta\equiv X^l\,(\mathrm{mod}\,\mathcal{P})$ has no solution in $L$.
\item[-] If $l\nmid a$, then $\mathcal{P}$ is totally ramified in $L(\sqrt[l]{\theta})$.
\end{itemize}
$(2)$ Let $\mathcal{B}$ a prime factor of $1-\zeta$ that divides $1-\zeta$ exactly to the a$^{th}$ power. Suppose that $\mathcal{B}\nmid\theta$, then $\mathcal{B}$ splits completely in $L(\sqrt[l]{\theta})$ if the congruence
\begin{center}
$\theta\equiv X^l\,(\mathrm{mod}\,\mathcal{B}^{al+1})$ \hspace{2cm} $(*)$ 
\end{center} 
has a solution in $L$. The ideal $\mathcal{B}$ is inert in $L(\sqrt[l]{\theta})$ if the congruence
\begin{center}
$\theta\equiv X^l\,(\mathrm{mod}\,\mathcal{B}^{al})$ 
\hspace{2cm} $(**)$
\end{center}
has a solution in $L$, but $(*)$ has none. The ideal $\mathcal{B}$ is totally ramified in $L$ if the congruence $(**)$ has no solution.
\end{proposition}
\begin{proof}
see [\cite{Hec} Theorem $118$, Theorem $119$]
\end{proof}
\section{\Large Prime factorization in pure quintic field and in its normal closure}
Let $\Gamma$, $k_0$ and $k$ as above. We begin with the factorization of primes in $\Gamma$, next we give the decomposition of primes in $k_0$, knowing that $5$ is the unique prime ramified in $k_0$. Finally the results of Proposition \ref{prKummer} allow us to state the prime factorization in $k$. For more details on decomposition laws, we refer the reader to \cite{Hec}, $\,$\cite{Irland} $\,$ \cite{cyc}.\\
Let $p\,\neq\, 5$ a prime number, $\mathcal{P}$ a prime ideal of $\Gamma$, $\pi$ a prime of $k_0$ and $\mathcal{L}$ a prime ideal of $k$. We denote by $\mathcal{N}$ the absolute norm.
\begin{proposition}\label{ramifGamm} Using the same notations as above, we have:
\item[$(1)$] If $p$ divides the radicand $n$ then $p\mathcal{O}_\Gamma = \mathcal{P}^5$ and $\mathcal{N}(\mathcal{P}) = p$.

\item[$(2)$] If $p\nmid 5n$ and $p\equiv\pm2\,(\mathrm{mod}\,5)$ then $p\mathcal{O}_\Gamma = \mathcal{P}_1\mathcal{P}_2$ and $\mathcal{N}(\mathcal{P}_1) = p$, $\mathcal{N}(\mathcal{P}_2) = p^4$.

\item[$(3)$] If $p\nmid 5n$ and $p\equiv-1\,(\mathrm{mod}\,5)$ then $p\mathcal{O}_\Gamma = \mathcal{P}_1\mathcal{P}_2\mathcal{P}_3$ and $\mathcal{N}(\mathcal{P}_1) = p$, $\mathcal{N}(\mathcal{P}_2) = \mathcal{N}(\mathcal{P}_3) = p^2$.

\item[$(4)$] If $p\nmid 5n$ and $p\equiv1\,(\mathrm{mod}\,5)$ then:
\begin{itemize}

\item[-]$p\mathcal{O}_\Gamma = \mathcal{P}_1\mathcal{P}_2\mathcal{P}_3\mathcal{P}_4\mathcal{P}_5$ and $\mathcal{N}(\mathcal{P}_i) = p$, if $n$ is quintic residue modulo $p$.
\item[-]$p\mathcal{O}_\Gamma = \mathcal{P}$ and $\mathcal{N}(\mathcal{P}_i) = p^5$, if $n$ is not a quintic residue modulo $p$.
 \end{itemize}
\end{proposition}
\begin{proof}
For $(1)$ see [\cite{Kobayashi}, Lemma 3].\\
Since the discriminant of $\Gamma/\mathbb{Q}$ is $disc(\Gamma/\mathbb{Q}) = 5^5n^4$ then if $p\nmid 5n$ we have that $p$ is unramified in $\Gamma$. Then the decomposition of $p\mathcal{O}_\Gamma$ is governed by the factorization of $X^5-n$ over the field 
$\mathbb{F}_p$ of $p$ elements.\\
- If $p\equiv1\,(\mathrm{mod}\,5)$ then $\mathbb{F}_p$ has five fifth roots of unity. Therefore either $X^5-n$ splits into five linear factors, or is irreducible, according to whether or not $n$ is a quintic residue modulo $p$. So in these two cases, $p\mathcal{O}_\Gamma$ splits into five ideals of norm $p$, or it remains inert of norm $p^5$.\\
-If $p\,\not\equiv\,1\,(\mathrm{mod}\,5)$ then $X^5-n$ has a unique linear factor over $\mathbb{F}_p$. Let $X-a$ be the linear factor of $X^5-n$ over $\mathbb{F}_p$. The remaining linear factors are $X-\zeta_5^{k}a$ where $k \in \{1,2,3,4\}$.\\
If $p\equiv-1\,(\mathrm{mod}\,5)$ then $(X-\zeta_5^2a)(X-\zeta_5^3a)$ and $(X-\zeta_5a)(X-\zeta_5^4a)$ are irreducible quadratics over $\mathbb{F}_p$. So $p\mathcal{O}_\Gamma$ is a product of a prime ideal of norm $p$ and two of norm $p^2$\\
If $p\equiv\pm2 \,(\mathrm{mod}\,5)$ then the $\zeta_5^ka$ are all conjugates over $\mathbb{F}_p$. Then $(X-\zeta_5a)(X-\zeta_5^2a)(X-\zeta_5^3a)(X-\zeta_5^4a)$ is irreducible over $\mathbb{F}_p$. So $p\mathcal{O}_\Gamma$ is a product of a prime ideal of norm $p$ and one of norm $p^4$.
\end{proof}
The ramification of the prime $5$ needs a particular treatment.
\begin{theorem}\label{fieldKind}
Using the same notations as above. Let $f$ be the conductor of $k/k_0$, and $R = q_1...q_s$ denotes the square free product of all prime divisors of the radicand $n$ of $\Gamma$, then $f$ satisfies the relation:

\begin{equation*}
f^4 =
\left\{
\begin{array}{rl}
5^2R^4 & \text{if } n^4\,\not\equiv\,1\,(\text{mod}\,25)\,\, \text{(field of the first kind)},\\
R^4 & \text{if } n^4\equiv1\,(\text{mod}\,25)\,\, \text{(field of the second kind)}.
\end{array}
\right.
\end{equation*}

\end{theorem}
\begin{proof}
see \cite{conductor}.
\end{proof}
\begin{proposition}\label{dec5GAmma} Using the same notations as above.  The decomposition into prime factors of $5$ is:
\begin{itemize}
\item[-] If $\Gamma$ is of first kind then: $5\mathcal{O}_\Gamma = \mathcal{P}^5$ and $\mathcal{N}(\mathcal{P}) = 5$.
\item[-] If $\Gamma$ is of second kind then: $5\mathcal{O}_\Gamma = \mathcal{P}_1\mathcal{P}_2^4$ and $\mathcal{N}(\mathcal{P}_i) = 5$.
\end{itemize}
\end{proposition}

\begin{proof}
See the proof of [\cite{Mayer}, Theorem 1.1].
\end{proof}

The decomposition law in the $5^{th}$ cyclotomic field $k_0$ is as follows:
\begin{proposition}\label{corocyc}
Using the same notations as above, we have :
\item[-] $5\mathcal{O}_{k_0} = \lambda^4 = (1-\zeta_5)^4$.
\item[-] $p\mathcal{O}_{k_0} = \pi_1\pi_2\pi_3\pi_4$ in $k_0$, if $p \, \equiv \, 1\, (\mathrm{mod}\, 5)$.
\item[-] $p\mathcal{O}_{k_0} = \pi_1\pi_2$ in $k_0$, if $p \, \equiv \, -1\, (\mathrm{mod}\, 5)$.
\item[-] $p\mathcal{O}_{k_0} = p$ in $k_0$ (inert), if $p \, \equiv \, \pm2\, (\mathrm{mod}\, 5)$.
\end{proposition}

\begin{proof}
Its follows from [\cite{cyc}, Theorem 2.13].
\end{proof}
Next, let $k$ be the normal closure of $\Gamma$. The decomposition of the prime $5$ in $k$ is the purpose of the following proposition.
\begin{proposition}\label{dec5k}
The prime $5$ decomposes in $k$ as follows:
\item[-] If $\Gamma$ is of first kind then: $5\mathcal{O}_k = \mathcal{L}^{20}$.
\item[-] If $\Gamma$ is of second kind then: $5\mathcal{O}_k = \mathcal{L}_1^4\mathcal{L}_2^4\mathcal{L}_3^4\mathcal{L}_4^4\mathcal{L}_5^4$.
\end{proposition}
\begin{proof}
We have $5$ ramifies in $k_0 = \mathbb{Q}(\zeta_5)$ then:
\item[-]Suppose that $\Gamma$ is of first kind, by Proposition \ref{dec5GAmma} we have $5\mathcal{O}_\Gamma = \mathcal{P}^5$. Hence $5\mathcal{O}_k = \mathcal{L}^{20}$.
\item[-]Suppose that $\Gamma$ is of second kind, by Proposition \ref{dec5GAmma} we have $5\mathcal{O}_\Gamma = \mathcal{P}_1^4\mathcal{P}_2$. It follows that,\\ $5\mathcal{O}_k = \mathcal{L}_1^4\mathcal{L}_2^4\mathcal{L}_3^4\mathcal{L}_4^4\mathcal{L}_5^4$.
\end{proof}
However, we have the following proposition in which we characterize the decomposition into prime ideals over $p\,\neq\,5$ in $k$.
\begin{proposition}\label{propk}
Using the same notations as above, we have :
\item[$(1)$] If $p$ divides $disc(\Gamma/\mathbb{Q})$ then: 
\begin{itemize}
\item[(a)]If $p \, \equiv \, \pm2\, (\mathrm{mod}\, 5)$, then $p\mathcal{O}_k = \mathcal{L}^5$.
\item[(b)] If $p \, \equiv \, -1\, (\mathrm{mod}\, 5)$, then $p\mathcal{O}_k = \mathcal{L}_1^5\mathcal{L}_2^5$.
\item[(c)] If $p \, \equiv \, 1\, (\mathrm{mod}\, 5)$, then $p\mathcal{O}_k = \mathcal{L}_1^5\mathcal{L}_2^5\mathcal{L}_3^5\mathcal{L}_4^5$.
\end{itemize}

\item[$(2)$] If $p$ does not divide $disc(\Gamma/\mathbb{Q})$  and $p \, \equiv \, 1\, (\mathrm{mod}\, 5)$ then: 
\begin{itemize}
\item[(a)] $p$ decomposes completely in $k$ if and only if $disc(\Gamma/\mathbb{Q})$ is a quintic residue modulo $p$.
\item[(b)] $p\mathcal{O}_k = \mathcal{L}_1\mathcal{L}_2\mathcal{L}_3\mathcal{L}_4$ if and only if $disc(\Gamma/\mathbb{Q})$ is not a quintic residue modulo $p$.
\end{itemize}

\item[$(3)$] If $p$ does not divide $disc(\Gamma/\mathbb{Q})$  and $p \, \equiv \, \pm2\, (\mathrm{mod}\, 5)$ then: $p\mathcal{O}_k = \mathcal{L}_1\mathcal{L}_2\mathcal{L}_3\mathcal{L}_4\mathcal{L}_5$.

\item[$(4)$] If $p$ does not divide $disc(\Gamma/\mathbb{Q})$  and $p \, \equiv \, -1\, (\mathrm{mod}\, 5)$ then: $p\mathcal{O}_k = \mathcal{L}_1\mathcal{L}_2\mathcal{L}_3\mathcal{L}_4\mathcal{L}_5\mathcal{L}_6\mathcal{L}_7\mathcal{L}_8\mathcal{L}_9\mathcal{L}_{10}$.
\end{proposition}
\begin{proof}
\item[$(1)$] We use 1 of Proposition \ref{ramifGamm} and the decomposition of prime ideals in the cyclotomic
field $k_0 =  \mathbb{Q}(\zeta_5)$.

\item[$(2)$] Suppose that $p$ does not divide $disc(\Gamma/\mathbb{Q})$ and $p \,\equiv \,1\, (\mathrm{mod}\, 5)$.
\begin{itemize}
\item[(a)] If $disc(\Gamma/\mathbb{Q})$ is a quintic residue modulo $p$, then by 4 of proposition \ref{ramifGamm} we have $p$ splits completely in $\Gamma$, and also in $k_0$. Hence $p$ splits completely in $k$.
\item[(b)] If $disc(\Gamma/\mathbb{Q})$ is not a quintic residue modulo $p$, we have $p$ remains inert in $\Gamma$. Hence $p\mathcal{O}_k = \mathcal{L}_1\mathcal{L}_2\mathcal{L}_3\mathcal{L}_4$.
\end{itemize}

\item[$(3)$] Since $p\nmid disc(\Gamma/\mathbb{Q})$ we have that $p$ is unramified in $\Gamma$, and also in $k_0$, then $p$ is unramified in $k$. Since $k/\mathbb{Q}$ is Galois extension,
if $p\mathcal{O}_k = \mathcal{L}_1....\mathcal{L}_r$ then $rf = 20$, with $f = [\mathcal{O}_k/\mathcal{L}_i : \mathbb{F}_p]$ the inertia degree. It is known that $f$ is multiplicative in towers of number fields, which implies that $f$ is a multiple of $4$ if $p \,\equiv \,\pm2\, (\mathrm{mod}\, 5)$, then $f = 4$ and $r = 5$. Otherwise $f = 20$ and $r = 1$, which means that $p$ remains inert in $k$, that is impossible because according to Proposition \ref{dec5GAmma}, $p$ decomposes in $\Gamma$ as $p\mathcal{O}_\Gamma = \mathcal{P}_1\mathcal{P}_2$.

\item[$(4)$]As the previous point, if $p \,\equiv \,-1\, (\mathrm{mod}\, 5)$ we have that  $f$ is a multiple of $2$, then $f = 2$ and $r = 10$. Otherwise $f = 10$ and $r = 2$, which means that $p\mathcal{O}_k = \mathcal{L}_1\mathcal{L}_2$. Since $p\mathcal{O}_\Gamma = \mathcal{P}_1\mathcal{P}_2\mathcal{P}_3$ and each $\mathcal{P}_i$ is unramified in $k$, furthermore $k/\Gamma$ is Galois, then $r_{\mathcal{P}_i}f_{\mathcal{P}_i} = 4$. For all possible values of $r_{\mathcal{P}_i}$ and $f_{\mathcal{P}_i}$ we can not get $p\mathcal{O}_k = \mathcal{L}_1\mathcal{L}_2$. Hence $f = 2$ and $r = 10$. 
\end{proof}
\begin{remark}
Let $p\,\neq\,5$, according to Proposition \ref{propk}, if $p\,\nmid\,disc(\Gamma/\mathbb{Q})$, then $p$ is unramified in $k$ and not inert in $k$. Let $p\mathcal{O}_k = \mathcal{L}_1\mathcal{L}_2....\mathcal{L}_r$, since the elements of $\mathrm{Gal}(k/\mathbb{Q})$ permute the set $\{\mathcal{L}_1....\mathcal{L}_r\}$ then $\mathcal{L}_i^\sigma\,\neq\,\mathcal{L}_i$ for $1\leq i \leq r$. This result has great importance in the proof of our main theorem.
\end{remark}

\section{Norm residue symbol}
Let $L/K$ be an abelian extension of number fields with conductor $f$. For each finite or infinite prime ideal $\mathcal{P}$ of $K$, we note by $f_{\mathcal{P}}$ the largest power of $\mathcal{P}$ that divides $f$. Let
$\beta \in K^{*}$, we determine an auxiliary number $\beta_0$ by the two conditions $\beta_0\equiv\beta\,(\mathrm{mod}\, f_\mathcal{P})$ and $\beta_0\equiv1\,(\mathrm{mod}\, \frac{f}{f_\mathcal{P}})$. Let $\mathcal{Q}$ an ideal co-prime with $\mathcal{P}$ such that $(\beta_0) = \mathcal{P}^a\mathcal{Q}$ ($a=0$ if $\mathcal{P}$ infinite). We note by
\begin{center}
\Large{ $\left(\frac{\beta,L}{\mathcal{P}} \right) =  \left(\frac{L/K}{\mathcal{Q}} \right)$}
\end{center}
the Artin map in $L/K$ applied to $\mathcal{Q}$.\\
Let $K$ be a number field containing the $m^{th}$-roots of unity, where $m\in \mathbb{N}$, then for each $\alpha,\beta \in K^{*}$ and prime ideal $\mathcal{P}$ of $K$, we define the norm residue symbol by:
\begin{center}
\Large{ $\left(\frac{\beta,\alpha}{\mathcal{P}} \right)_m =  \frac{\left(\frac{\beta,K(\sqrt[m]{\alpha})}{\mathcal{P}} \right)\sqrt[m]{\alpha}}{\sqrt[m]{\alpha}}$}
\end{center}
Therefore, if the prime ideal $\mathcal{P}$ is unramified in the field $K(\sqrt[m]{\alpha})$, then we write:
\begin{center}
\Large{ $\left(\frac{\alpha}{\mathcal{P}} \right)_m =  \frac{\left(\frac{K(\sqrt[m]{\alpha})}{\mathcal{P}} \right)\sqrt[m]{\alpha}}{\sqrt[m]{\alpha}}$}
\end{center}
\begin{remark}
\large{Notice that $\left(\frac{\beta,\alpha}{\mathcal{P}} \right)_m$ and $\left(\frac{\alpha}{\mathcal{P}} \right)_m$} are two $m^{th}$-roots of unity.
\end{remark}
\large{Following \cite{Hass}, the principal properties of the norm residue symbol are given as follows:}
\begin{Properties}.\\\label{normprop}
\Large{\item[$(1)$] $\left(\frac{\beta_1\beta_2,\alpha}{\mathcal{P}} \right)_m = \left(\frac{\beta_1,\alpha}{\mathcal{P}} \right)_m\left(\frac{\beta_2,\alpha}{\mathcal{P}} \right)_m$};

\Large{\item[$(2)$] $\left(\frac{\beta,\alpha_1\alpha_2}{\mathcal{P}} \right)_m = \left(\frac{\beta,\alpha_1}{\mathcal{P}} \right)_m\left(\frac{\beta,\alpha_2}{\mathcal{P}} \right)_m$};

\Large{\item[$(3)$] $\left(\frac{\beta,\alpha}{\mathcal{P}} \right)_m = \left(\frac{\alpha,\beta}{\mathcal{P}} \right)_m^{-1}$};

\item[$(4)$] If $\mathcal{P}$ is not divisible by the conductor $f(\sqrt[m]{\alpha})$ of $K(\sqrt[m]{\alpha})$ and appears in $(\beta)$ with the exponent b, then: 
\Large{$\left(\frac{\beta,\alpha}{\mathcal{P}} \right)_m = \left(\frac{\alpha}{\mathcal{P}} \right)_m^{-b}$  };

\Large{\item[$(5)$] $\left(\frac{\beta,\alpha}{\mathcal{P}} \right)_m = 1$ if and only if $\beta$ is norm residue of $K(\sqrt[m]{\alpha})$ modulo $f(\sqrt[m]{\alpha})$  };

\Large{\item[$(6)$] $\left(\frac{\tau\beta,\tau\alpha}{\tau\mathcal{P}} \right)_m = \tau\left(\frac{\beta,\alpha}{\mathcal{P}} \right)_m$ for each automorphism $\tau$ of $K$ };

\Large{\item[$(7)$] ${\displaystyle \prod_{\mathcal{P}} \left(\frac{\beta,\alpha}{\mathcal{P}} \right)_m} = 1$} for all finite or infinite prime ideals;

\item[$(8)$]If $K'$ is a finite extension of $K$, $\alpha \in K^{*},\beta' \in K'$ then: 
\begin{center}
\Large{${\displaystyle \prod_{\mathcal{P'}|\mathcal{P}} \left(\frac{\beta',\alpha}{\mathcal{P'}} \right)_m} =  \left(\frac{\mathcal{N}_{K'/K}(\beta'),\alpha}{\mathcal{P}} \right)_m$}
\end{center}

\item[$(9)$]Let $\alpha,\beta \in K^{*}$ and the conductors $f(\sqrt[m]{\alpha})$, $f(\sqrt[m]{\beta})$ of respectively $K(\sqrt[m]{\alpha})$, $K(\sqrt[m]{\beta})$ are co-prime then, the classical reciprocity law:
\begin{center}
\Large{$\left(\frac{\beta}{(\alpha)} \right)_m = \left(\frac{\alpha}{(\beta)} \right)_m$}
\end{center}

\end{Properties}

\large{For more basic properties of the quintic norm residue symbols in the number fields, we refer the
reader to \cite{Hass}.\\
Notice that in the rest of the article, we will use the quintic norm residue symbol $(m = 5)$. As the ring of integer $\mathcal{O}_{k_0}$ is principal, we will write the norm quintic residue symbol as follows:}
\begin{center}
\Large{$\left(\frac{\beta,\alpha}{(\pi)} \right)_5 = \left(\frac{\beta,\alpha}{\pi} \right)_5$ and $\left(\frac{\alpha}{(\pi)} \right)_5 = \left(\frac{\alpha}{\pi} \right)_5$}
\end{center}
Where $\alpha,\beta \in k_0^{*}$ and $\pi$ is a prime integer of $\mathcal{O}_{k_0}$.

\section{\Large Fields $\mathbb{Q}(\sqrt[5]{n},\zeta_5)$ whose $5$-class group is of type $(5,5)$}
In this section we use the class number formula of [\cite{Pa}] to give  a necessary and sufficient condition such that the $5$-class group of the fields $\mathbb{Q}(\sqrt[5]{n},\zeta_5)$ is of type $(5,5)$.\\
\begin{lemma}\label{C+C-}
Let $\omega \in \mathrm{Gal}(k/\Gamma) = \langle\tau\rangle\,\simeq \,\mathbb{Z}/4\mathbb{Z}$, and $C$ be a $\mathbb{Z}_5[\langle\tau\rangle]$ module. Let $C^+ = \{\mathcal{A}\in C\,|\, \mathcal{A}^{\omega} = \mathcal{A} \}$ and $C^- = \{\mathcal{A}\in C\,|\, \mathcal{A}^{\omega} = \mathcal{A}^{-1}\}$. Then
\begin{center}
$C\,\cong\, C^+\times C^-$.
\end{center}
\end{lemma}
\begin{proof}
Let $\mathcal{A} \in C$. Write $\mathcal{A} = \mathcal{A} ^{\frac{1+\tau^2}{2}}.\mathcal{A}^{\frac{1-\tau^2}{2}}$. Then $\mathcal{A}^{\frac{1+\tau^2}{2}}\in C^+$ and $\mathcal{A}^{\frac{1-\tau^2}{2}}\in C^-$.\\ 
Let $\mathcal{A} \in C^+ \cap C^-$, then there are $\omega_1, \, \omega_2 \in \mathrm{Gal}(k/\Gamma)$ such that $\mathcal{A}^{\omega_1} = \mathcal{A}$ and $\mathcal{A}^{\omega_2} = \mathcal{A}^{-1}$. If $\omega_1 = \omega_2$ then we have $\mathcal{A} = \mathcal{A}^{-1}$ that is $\mathcal{A}^2 = 1$. Thus $\mathcal{A} = 1$, since $C$ is $\mathbb{Z}_5[<\tau>]$-module. If $\omega_1 \neq \omega_2$ then by treatment of all possible cases, its easy to show that $\mathcal{A} = 1$. Hence $C$ $\cong$ $C^+ \times C^-$.
\end{proof}
\begin{lemma}\label{Gamma C+}
$C_{\Gamma,5}\,\cong\, C_{k,5}^{+}$.
\end{lemma}
\begin{proof}
We admit the same proof as [\cite{Mani}, Lemma 6.2].
\end{proof}

Now let $u$ be the index of the subgroup $E_0$ generated by the units of intermediate fields of the extension $k/\mathbb{Q}$ in the unit group of $k$. In [\cite{Pa}], C.Parry proved that $u$ is a divisor of $5^6$, and he presented the relation formula between the class numbers of $k$ and $\Gamma$ as follows: $h_{k} = (\frac{u}{5})(\frac{h_\Gamma}{5})^{4}$. The structure of the $5$-class group $C_{k,5}$ is given by the following proposition:

\begin{proposition}\label{prop55} Let $\Gamma$ be a pure quintic field, $k$ its normal closure, and $u$ be the index of units defined above, then

\item \hspace{3cm}$C_{k,5} \simeq \mathbb{Z}/5\mathbb{Z} \times \mathbb{Z}/5\mathbb{Z} \, \Longleftrightarrow\, h_\Gamma$ is exactly divisible by $5$ and $u  =  5^3$.

\end{proposition}
\begin{proof}
\item Let $C_{k,5} \simeq \mathbb{Z}/5\mathbb{Z} \times \mathbb{Z}/5\mathbb{Z}$ then $|C_{k,5}| = 25$. According to [\cite{Pa}]  $|C_{k,5}|\, = \, (\frac{u}{5})\frac{|C_{\Gamma,5}|^{4}}{5^4}\, = \, 25$, namly $uh_\Gamma^4 = 5^7$. Let $n = v_5(u)$, and $n' = v_5(h_\Gamma)$, so we get $n+4n' = 7$, and the unique natural values of $n$ and $n'$ which satisfy this equation are $n = 3$ and $n' = 1$. Therfore we have $u = 5^3$ and $h_\Gamma$ is exactly divisible by $5$.\\ 
Conversely according to Lemma \ref{C+C-} and Lemma \ref{Gamma C+}, $C_{k,5}\,\cong \,C_{k,5}^{+}\times C_{k,5}^{-}$ and $C_{k,5}^{+}\,\cong\, C_{\Gamma,5}$, then $C_{k,5}^{+}$ is cyclic of order $5$ because $|C_{k,5}^{+}| = |C_{\Gamma,5}| =  5$. Since $|C_{k,5}| = |C_{k,5}^{+}|.|C_{k,5}^{-}| =  25$ by the formula $|C_{k,5}|\, = \, (\frac{u}{5})\frac{|C_{\Gamma,5}|^{4}}{5^4}$, we get that $|C_{k,5}^{-}| =  5$ and $C_{k,5}^{-}$ is a cyclic group of order $5$, then $C_{k,5} \,\simeq \, \mathbb{Z}/5\mathbb{Z} \times \mathbb{Z}/5\mathbb{Z}$.
\end{proof}

\section{\Large Proof of the Main Theorem}
We start the proof by some results as follows.
\begin{proposition}\label{propprof}
Let $p$ a prime number such that $p\equiv-1\, (\mathrm{mod}\, 5)$ and $\mathrm{Gal}(k/\Gamma) = \langle\tau\rangle$, then we have:
\begin{itemize}
\item[$(1)$] $p\,=\pi_1\pi_2$, where $\pi_i$ are primes of $k_0$ and $\pi_1^{\tau} = -\pi_2$,  $\pi_2^{\tau} = -\pi_1$
\item[$(2)$] \large{$\left(\frac{c}{\pi_1} \right)_5 = \left(\frac{c}{\pi_2} \right)_5^2$}  for all $c \in \mathbb{Z}$ such that $p\nmid c$.
\item[$(3)$] \large{$\left(\frac{c}{\pi_1} \right)_5 = \left(\frac{c}{\pi_2} \right)_5 = 1$} if and only if $c$ is a quintic residue modulo $p$.
\item[$(4)$] \large{$\left(\frac{\pi_2}{\pi_1} \right)_5 = \left(\frac{\pi_1}{\pi_2} \right)_5 = 1$}.
\end{itemize}
\end{proposition}
\begin{proof}
\begin{itemize}
\item[$(1)$]According to [\cite{Bahmanpour}, lemma 3.1] the prime $p\equiv-1\, (\mathrm{mod}\, 5)$, can be written as:\\ $p = a^2+ab-b^2$, with $a,b\in \mathbb{Z}$ co-prime. Using this expression of $p$ we define:
\begin{center}
$\pi_1 = a\zeta_5^3+a\zeta_5^2+b$\hspace{1cm} and \hspace{1cm} $\pi_2 = a\zeta_5^3+a\zeta_5^2+a-b$, 
\end{center}
then $p = \pi_1\pi_2$ (see [\cite{Mani}, theorem 5.15]). Let $\mathrm{Gal}(k/\Gamma) = \langle\tau\rangle$ with $\tau:\,\zeta_5\longrightarrow\zeta_5^2$, so we get that:\\
$\pi_1^{\tau} = (a\zeta_5^3+a\zeta_5^2+b)^\tau = a\zeta_5+a\zeta_5^4+b = a(\zeta_5+\zeta_5^4)+b = a(-1-\zeta_5^2-\zeta_5^3)+b = -\pi_2$.\\
$\pi_2^{\tau} = (a\zeta_5^3+a\zeta_5^2+a-b)^\tau = a\zeta_5+a\zeta_5^4+a-b = a(\zeta_5+\zeta_5^4)+a-b = a(-1-\zeta_5^2-\zeta_5^3)+a-b = -\pi_1$.
\item[$(2)$] \large{$\left(\frac{c}{\pi_1} \right)_5 = \left(\frac{\pi_1,c}{\pi_1} \right)_5^{-1} = \left(\frac{c,\pi_1}{\pi_1} \right)_5 = \left(\frac{c,-\pi_2^\tau}{-\pi_2^\tau} \right)_5 = \tau\left(\frac{c,\pi_2}{\pi_2} \right)_5 = \left(\frac{c,\pi_2}{\pi_2} \right)_5^2 = \left(\frac{\pi_2,c}{\pi_2} \right)_5^{-2} = \left(\frac{c}{\pi_2} \right)_5^2$}, because $\pi_1$, $\pi_2$ are unramified in $k_0(\sqrt[5]{c})$ (proposition \ref{prKummer}), and we use properties $(3),(4), (6)$ of \ref{normprop}.
\item[$(3)$]Let $\omega \in \mathcal{O}_{k_0}$ and $\pi$ a prime element of $\mathcal{O}_{k_0}$, such that $\pi \nmid\omega$, then $X^5\equiv\omega\,(\mathrm{mod\,\pi})$ is soluble in $\mathcal{O}_{k_0}$ if and only if $\omega^m\equiv 1\,(\mathrm{mod\,\pi})$, when $m = \frac{\mathcal{N}(\pi)-1}{5}$. If \large{$\left(\frac{c}{\pi_1} \right)_5 = \left(\frac{c}{\pi_2} \right)_5 = 1$}, we get that $\pi_1$ and $\pi_2$ split completely in $k_0(\sqrt[5]{c})$, so the equations $X^5\equiv c\,(\mathrm{mod\,\pi_1})$ and $X^5\equiv c\,(\mathrm{mod\,\pi_2})$ are soluble in $\mathcal{O}_{k_0}$, therefore $c^{\frac{p^2-1}{5}}\equiv 1\,(\mathrm{mod\,\pi_1})$ and $c^{\frac{p^2-1}{5}}\equiv 1\,(\mathrm{mod\,\pi_2})$, thus $c^{\frac{p^2-1}{5}}\equiv 1\,(\mathrm{mod\,p})$, and by Euler's citernion, $c$ is a quintic residue modulo $p$. Conversely if $c$ is rest quintic modulo $p$ we have $c^{\frac{p^2-1}{5}}\equiv 1\,(\mathrm{mod\,p})$, so $c^{\frac{p^2-1}{5}}\equiv 1\,(\mathrm{mod\,\pi_1})$ and $c^{\frac{p^2-1}{5}}\equiv 1\,(\mathrm{mod\,\pi_2})$, therefore the equations $X^5\equiv c\,(\mathrm{mod\,\pi_1})$ and $X^5\equiv c\,(\mathrm{mod\,\pi_2})$ are soluble in $\mathcal{O}_{k_0}$, namly \large{$\left(\frac{c}{\pi_1} \right)_5 = \left(\frac{c}{\pi_2} \right)_5 = 1$}.
\item[$(4)$] Accoding to [\cite{Mani}, theorem 5.15] we have $A = $\large{$\left(\frac{\pi_1^{a_1}\pi_2^{a_2},\pi_1,\pi_2}{\pi_1} \right)_5 = 1$} with $a_1, a_2 \in \{1,2,3,4\}$ and $a_1 \neq a_2$, so using properties \ref{normprop} we have 
\begin{center}
$A$  = \large{$\left(\frac{\pi_1,\pi_1}{\pi_1}\right)_5^{a_1}.\left(\frac{\pi_1,\pi_2}{\pi_1}\right)_5^{a_1}$.$\left(\frac{\pi_2,\pi_1}{\pi_1}\right)_5^{a_2}$.$\left(\frac{\pi_2,\pi_2}{\pi_1}\right)_5^{a_2}$} = $1$
\end{center}
then
\begin{itemize}
\item[-] \large{$\left(\frac{\pi_1,\pi_1}{\pi_1}\right)_5^{a_1} = 1$}, because $\pi_1$ is norm in $k_0(\sqrt[5]{\pi_1})/k_0$.

\item[-] \large{$\left(\frac{\pi_1,\pi_2}{\pi_1}\right)_5^{a_1} = \left(\frac{\pi_2}{\pi_1}\right)_5^{-a_1}$} and \large{$\left(\frac{\pi_2,\pi_1}{\pi_1}\right)_5^{a_2} = \left(\frac{\pi_2}{\pi_1}\right)_5^{a_2}$}, by $(4)$ of properties \ref{normprop}

\item[-] \large{$\left(\frac{\pi_2,\pi_2}{\pi_1}\right)_5^{a_2} = 1$}, by $(2)$ and $(4)$ of properties \ref{normprop}.
\end{itemize}
We get that, \large{$A = \left(\frac{\pi_2}{\pi_1}\right)_5^{a_2-a_1} = 1$}, then $a_2-a_1 = 0$ or $5$, which is impossible because $a_1, a_2 \in \{1,2,3,4\}$ and $a_1 \neq a_2$, hence \large{$\left(\frac{\pi_2}{\pi_1}\right)_5 = 1$}. Since $\pi_1$ and $\pi_2$ play symmetric roles we have also \large{$\left(\frac{\pi_1}{\pi_2} \right)_5 = 1$}. Thus we deduce that \large{$\left(\frac{\pi_2}{\pi_1} \right)_5 = \left(\frac{\pi_1}{\pi_2} \right)_5 = 1$}.
\end{itemize}
\end{proof}
\subsection{Case 1: $n = 5^{e_1}p^{e_2}\,\not\equiv\,\pm1,\pm7\,(\mathrm{mod}\,25)$ with $p\,\not\equiv\,-1\,(\mathrm{mod}\,25)$}\label{case1}
\large{Let $\Gamma = \mathbb{Q}(\sqrt[5]{5^{e_1}p^{e_2}})$ be a pure quintic field where $e_1, e_2 \in\{1,2,3,4\}$, $p$ a prime number such that $p\,\not\equiv\,-1\,(\mathrm{mod}\,25)$ and $k = \Gamma(\zeta_5)$ its normal closure. We have $p = \pi_1\pi_2$ in $k_0$. By proposition \ref{ramifGamm}, $p$ is totally ramified in $\Gamma$, therefore $\pi_1,\,\pi_2$ are ramified in $k$. According to [\cite{Mani}, Lemma 5.1] we have $\lambda = 1-\zeta_5$ is ramified in $k/k_0$ because $5^{e_1}p^{e_2}\,\not\equiv \pm1 \pm7\,(\mathrm{mod}\,\lambda^5)$ $(e\neq0)$.\\
If  we denoteby $\mathcal{P}_1, \mathcal{P}_2$ and $\mathcal{I}$ respectively, the prime ideals of $k$ above $\pi_1, \pi_2$ and $\lambda$ respectively, we get that $\mathcal{P}_i^5 = \pi_i\mathcal{O}_k$ $(i=1,2)$, $\mathcal{I}^5 = \lambda\mathcal{O}_k$, $\mathcal{P}_i^\sigma = \mathcal{P}_i$, $\mathcal{P}_1^\tau = \mathcal{P}_2$, $\mathcal{I}^\sigma = \mathcal{I}^\tau = \mathcal{I}$. 
 }.\\
According to Proposition \ref{prop55}, we have $C_{k,5}\,\cong\,C_{k,5}^+\times C_{k,5}^-$, such that $C_{k,5}^+$ and $C_{k,5}^-$ are cyclic subgoups of order $5$. Since $\mathcal{P}_i^\sigma = \mathcal{P}_i$ $(i = 1,2)$, $\mathcal{P}_1^\tau = \mathcal{P}_2$, $\mathcal{P}_2^\tau = \mathcal{P}_1$ and $\mathcal{I}^\sigma = \mathcal{I}^{\tau} = \mathcal{I}$, we get that $[\mathcal{P}_1]$,$[\mathcal{P}_2]$ and $[\mathcal{I}]$ are ambiguous classes. As $C_{k,5}^{(\sigma)}$ is elementary group of rank $1$, we can deduce that $C_{k,5}^{(\sigma)} = \langle[\mathcal{P}_1]\rangle$. Also as $[\mathcal{P}_1]^{\tau^2} = [\mathcal{P}_1]$, we have $C_{k,5}^+ = \langle[\mathcal{P}_1]\rangle$, therefore $C_{k,5}^+ = C_{k,5}^{(\sigma)} = \langle[\mathcal{P}_1]\rangle$ if and only if $\mathcal{P}_1$ is not principal.\\
we argue by reduction to absurd: Assume that $\mathcal{P}_1$ is principal, we have:

\begin{center}
$[\mathcal{P}_1] = 1\Longrightarrow\, \exists \beta \in \mathcal{O}_{k}\, |\,\mathcal{P}_1 = \beta\mathcal{O}_{k}$ \\
\hspace{2.5cm}$\Longrightarrow\, \mathcal{N}_{k/k_0}(\mathcal{P}_1) = \mathcal{N}_{k/k_0}(\beta\mathcal{O}_{k})$\\
\hspace{2cm}$\Longrightarrow\, \pi_1\mathcal{O}_{k_0} = \mathcal{N}_{k/k_0}(\beta)\mathcal{O}_{k_0}$\\
\hspace{2.5cm} $\Longrightarrow\, \exists \epsilon \in E_{k_0}\, |\,\pi_1 = \epsilon\mathcal{N}_{k/k_0}(\beta)$ 
\end{center}

According to [\cite{Mani}, theorem 5.15] $E_{k_0}\subset \mathcal{N}_{k/k_0}(k^{*})$ because  $q^{*} = 2$ 
\begin{center}
\hspace{3cm}$\Longrightarrow\, \exists \alpha \in k-\{0\}\, |\,\pi_1 = \mathcal{N}_{k/k_0}(\alpha)$ 
\end{center}
that is to say $\pi_1$ is norm in $k = k_0(\sqrt[5]{5^e\pi_1\pi_2})$, where $\pi_1,\pi_2$ are primes above $p$ in $k_0$. Hence we have:
\begin{center}
\large{$\left(\frac{\pi_1,5^e\pi_1\pi_2}{\mathcal{P}}\right)_5 = 1$}
\end{center}
for all ideals $\mathcal{P}$ of $k_0$.\\
In particular, we calculate this symbol for $\mathcal{P} = \pi_1\mathcal{O}_{k_0}$ or $\mathcal{P} = \pi_2\mathcal{O}_{k_0}$. On the one hand,
\begin{center}
$A$  = \large{$\left(\frac{\pi_1,5^e\pi_1\pi_2}{\pi_1\mathcal{O}_{k_0}}\right)_5\,= \left(\frac{\pi_1,5^e}{\pi_1}\right)_5\,. \left(\frac{\pi_1,\pi_1}{\pi_1}\right)_5. \left(\frac{\pi_1,\pi_2}{\pi_1}\right)_5$}
\end{center}
On the other hand, we have:
\begin{itemize}
\item[-] \large{$\left(\frac{\pi_1,\pi_1}{\pi_1}\right)_5 = 1$}, because $\pi_1$ is norm in $k_0(\sqrt[5]{\pi_1})/k_0$.

\item[-] \large{$\left(\frac{\pi_1,\pi_2}{\pi_1}\right)_5 = \left(\frac{\pi_2}{\pi_1}\right)_5^{-1} = 1$}, by $(4)$ of properties \ref{normprop}, and $(4)$ of \ref{propprof}.

\item[-] \large{$\left(\frac{\pi_1,5^e}{\pi_1}\right)_5 = \left(\frac{5}{\pi_1}\right)_5^{-e}$}, by $(4)$ of properties \ref{normprop}.
\end{itemize}
We get that, \large{$A = \left(\frac{5}{\pi_1}\right)_5^{-e}=\,1$}. Since $\pi_1$ and $\pi_2$ play symmetric roles, then $B$ = \large{$\left(\frac{\pi_1,5^e\pi_1\pi_2}{\pi_2\mathcal{O}_{k_0}}\right)_5 = \left(\frac{5}{\pi_2}\right)_5^{-e}$}, 
and since \large{$\left(\frac{\pi_1,5^e\pi_1\pi_2}{\mathcal{P}}\right)_5 = 1$}, for all prime ideals of $k_0$, then $A = B = 1$, namly 
\begin{center}
\large{$\left(\frac{5}{\pi_1}\right)_5^{-e} = \left(\frac{5}{\pi_2}\right)_5^{-e} = 1$}. 
\end{center}
In fact that $5$ is not a quintic residue modulo $p$, implies that
\begin{center}
\large{$\left(\frac{5}{\pi_1\pi_2}\right)_5 = \left(\frac{5}{\pi_1}\right)_5\left(\frac{5}{\pi_2}\right)_5\,\neq\,1$}. 
\end{center}
then
\begin{center}
\large{$\left(\frac{5}{\pi_1}\right)_5\,\neq\,1\,\mathrm{or}\,\left(\frac{5}{\pi_2}\right)_5\,\neq\,1$}. 
\end{center}
Since $5$ does not divide $e$, then
\begin{center}
\large{$\left(\frac{5}{\pi_1}\right)_5^e\,\neq\,1\,\mathrm{or}\,\left(\frac{5}{\pi_2}\right)_5^e\,\neq\,1$}. 
\end{center}
Which is a contradiction. Consequently, the ideal $\mathcal{P}_1$ is not principal.\\
The second step in the proof of the case 1, is to find a non-ambiguous class, which generates the group $C_{k,5}^-$. By class field theory, the classes arise from unramified primes of $k_0$, these are not ambigous. Let $l$ be prime integer such that $l\,\neq\,p$, then $l$ is unramified in $\Gamma$. Let $\pi'$ a prime of $k_0$ above $l$, then $\pi'$ is unramified in $k$. Let $\mathcal{L}$ a prime ideal of $k$ above $\pi'$, then $C_{k,5}^- = \langle[\mathcal{L}]^{1-\tau^2}\rangle$ if and only if $\mathcal{L}$ is not principal.\\
we argue by reduction to absurd: Assume that $\mathcal{L}$ is principal, we have:

\begin{center}
$[\mathcal{L}] = 1\Longrightarrow\, \exists \beta \in \mathcal{O}_{k}\, |\,\mathcal{L} = \beta\mathcal{O}_{k}$ \\
\hspace{2.5cm}$\Longrightarrow\, \mathcal{N}_{k/k_0}(\mathcal{L}) = \mathcal{N}_{k/k_0}(\beta\mathcal{O}_{k})$\\
\hspace{2cm}$\Longrightarrow\, \pi'\mathcal{O}_{k_0} = \mathcal{N}_{k/k_0}(\beta)\mathcal{O}_{k_0}$\\
\hspace{2.5cm} $\Longrightarrow\, \exists \epsilon \in E_{k_0}\, |\,\pi' = \epsilon\mathcal{N}_{k/k_0}(\beta)$ 
\end{center}

By [\cite{Mani}, theorem 5.15] we have $E_{k_0}\subset \mathcal{N}_{k/k_0}(k^{*})$, then
\begin{center}
\hspace{3cm}$\Longrightarrow\, \exists \alpha \in k-\{0\}\, |\,\pi' = \mathcal{N}_{k/k_0}(\alpha)$ 
\end{center}
that is to say $\pi'$ is norm in $k = k_0(\sqrt[5]{5^e\pi_1\pi_2})$, where $\pi_1,\pi_2$ are two primes of $k_0$ such that $p = \pi_1\pi_2$. Hence we have:
\begin{center}
\large{$\left(\frac{\pi',5^e\pi_1\pi_2}{\mathcal{P}}\right)_5 = 1$}
\end{center}
for all ideals $\mathcal{P}$ of $k_0$.\\
In particular, we calculate this symbol for $\mathcal{P} = \pi_1\mathcal{O}_{k_0}$ or $\mathcal{P} = \pi_2\mathcal{O}_{k_0}$. On one hand,
\begin{center}
$A$  = \large{$\left(\frac{\pi',5^e\pi_1\pi_2}{\pi_1\mathcal{O}_{k_0}}\right)_5\,= \left(\frac{\pi',5^e}{\pi_1}\right)_5\,. \left(\frac{\pi',\pi_1}{\pi_1}\right)_5. \left(\frac{\pi',\pi_2}{\pi_1}\right)_5$}
\end{center}
On the other hand, we have:
\begin{itemize}
\item[-] \large{$\left(\frac{\pi',5^e}{\pi_1}\right)_5 = \left(\frac{\pi',5}{\pi_1}\right)_5^e = \left(\frac{\pi'}{\pi_1}\right)_5^{0\times e} = 1$}, by $(1)$, $(4)$ of properties \ref{normprop}.

\item[-] \large{$\left(\frac{\pi',\pi_2}{\pi_1}\right)_5 = \left(\frac{\pi'}{\pi_1}\right)_5^{0} = 1$}, by $(4)$ of properties \ref{normprop}.

\item[-] \large{$\left(\frac{\pi',\pi_1}{\pi_1}\right)_5 = \left(\frac{\pi'}{\pi_1}\right)_5$}, by $(4)$ of properties \ref{normprop}.
\end{itemize}
We get that, \large{$A = \left(\frac{\pi'}{\pi_1}\right)_5=\,1$}. Since $\pi_1$ and $\pi_2$ play symmetric roles, then $B$ = \large{$\left(\frac{\pi',5^e\pi_1\pi_2}{\pi_2\mathcal{O}_{k_0}}\right)_5 = \left(\frac{\pi'}{\pi_2}\right)_5$}, 
and since \large{$\left(\frac{\pi',5^e\pi_1\pi_2}{\mathcal{P}}\right)_5 = 1$}, for all prime ideals of $k_0$, then $A = B = 1$, namly 
\begin{center}
\large{$\left(\frac{\pi'}{\pi_1}\right)_5 = \left(\frac{\pi'}{\pi_2}\right)_5 = 1$}. 
\end{center}
In fact that $l$ is not a quintic residue modulo $p$, implies that
\begin{center}
\large{$\left(\frac{\pi'}{\pi_1\pi_2}\right)_5 = \left(\frac{\pi'}{\pi_1}\right)_5\left(\frac{\pi'}{\pi_2}\right)_5\,\neq\,1$}. 
\end{center}
then
\begin{center}
\large{$\left(\frac{ \pi'}{\pi_1}\right)_5\,\neq\,1\,\mathrm{or}\,\left(\frac{\pi'}{\pi_2}\right)_5\,\neq\,1$}. 
\end{center}
Which is a contradiction. Consequently, the ideal $\mathcal{L}$ is not principal.\\
It is easy to see that $[\mathcal{L}]^{1-\tau^2}\,\in\,C_{k,5}^-$, and $[\mathcal{L}]^{1-\tau^2}\,\neq\,1$, otherwise $[\mathcal{L}]^{1-\tau^2} = 1$ then $[\mathcal{L}] = [\mathcal{L}]^{\tau^2}$, therefore $[\mathcal{L}]\,\in\,C_{k,5}^+ = C_{k,5}^{(\sigma)}$, which contradict the fact that the class $[\mathcal{L}]$ is not ambigous.
Finally we deduce that
\begin{center}
$C_{k,5}\,\cong\,\langle[\mathcal{P}_1]\rangle\,\times \langle[\mathcal{L}]^{1-\tau^2}\rangle\,\cong\,\langle[\mathcal{P}_1],[\mathcal{L}]^{1-\tau^2}\rangle$
\end{center}
\begin{remark}
Since $[\mathcal{P}_2]$ and $[\mathcal{I}]$ are also ambiguous classes, we can prove by the same reasoning that:
\begin{center}
$C_{k,5}\,\cong\,\langle[\mathcal{P}_2],[\mathcal{L}]^{1-\tau^2}\rangle\,\cong\,\langle[\mathcal{I}],[\mathcal{L}]^{1-\tau^2}\rangle$
\end{center}
\end{remark}
\subsection{Case 2: $n = p^{e_1}q^{e_2}\equiv\pm1,\pm7\,(\mathrm{mod}\,25)$ with $p\,\not\equiv\,-1\,(\mathrm{mod}\,25)$ and $q\,\not\equiv\,\pm7\,(\mathrm{mod}\,25)$}\label{case2}
\large{Let $\Gamma = \mathbb{Q}(\sqrt[5]{p^{e_1}q^{e_2}})$ be a pure quintic field where $e\in\{1,2,3,4\}$, $p$, $q$  primes such that $p\,\not\equiv\,-1\,(\mathrm{mod}\,25)$ and $q\,\not\equiv\,\pm7\,(\mathrm{mod}\,25)$ and $k = \Gamma(\zeta_5)$ the normal closure. We have $p = \pi_1\pi_2$ and $q$ is inert in $k_0$ from Proposition \ref{corocyc}. By Proposition \ref{ramifGamm}, $p$ and $q$ are totally ramified in $\Gamma$, therefore $\pi_1,\,\pi_2$ and $q$ are ramified in $k$. According to [\cite{Mani}, Lemma 5.1] we have $\lambda = 1-\zeta_5$ is not ramified in $k/k_0$ because $n\,\equiv \pm1, \pm7\,(\mathrm{mod}\,\lambda^5)$.\\
If  we denote by $\mathcal{P}_1, \mathcal{P}_2, \mathcal{Q}$ respectively, the prime ideals of $k$ above $\pi_1, \pi_2, q$ respectively, we get that $\mathcal{P}_i^5 = \pi_i\mathcal{O}_k$ $(i=1,2)$, $\mathcal{Q}^5 = q\mathcal{O}_k$. Let $l$ be a prime integer such that $l\,\neq\,p$ and 
$l\,\neq\,q$, then $l$ is unramified in $\Gamma$. Let $\mathcal{L}$ be a prime ideal of $k$ above $l$.\\
We prove the result of the second point of the Main Theorem \ref{thp}, by the same reasoning as for case $1$. It is sufficient to replace $5$ by $q$.
Then we have,
\begin{center}
$C_{k,5}\,\cong\,\langle[\mathcal{P}_1]\rangle\,\times \langle[\mathcal{L}]^{1-\tau^2}\rangle\,\cong\,\langle[\mathcal{P}_1],[\mathcal{L}]^{1-\tau^2}\rangle$
\end{center}
 \begin{remark}
Since $[\mathcal{P}_2]$ and $[\mathcal{Q}]$ are also ambiguous classes, we can prove by the same reasoning that:
\begin{center}
$C_{k,5}\,\cong\,\langle[\mathcal{P}_2],[\mathcal{L}]^{1-\tau^2}\rangle\,\cong\,\langle[\mathcal{Q}],[\mathcal{L}]^{1-\tau^2}\rangle$
\end{center}
\end{remark}
\subsection{Case 3: $n = p^e\equiv\pm1,\pm7\,(\mathrm{mod}\,25)$ with $p\equiv-1\,(\mathrm{mod}\,25)$}
Let $\Gamma = \mathbb{Q}(\sqrt[5]{p^e})$ be a pure quintic field, where $e\in\{1,2,3,4\}$, and $p$ a prime such that, $p\equiv-1\,(\mathrm{mod}\,25)$, and $k = \Gamma(\zeta_5)$ its normal closure. Without loosing generality, we can choose $e = 1$. Since the field $\Gamma$ is of second kind, then by Proposition \ref{dec5k}, $5$ decomposes in $k$ as $5\mathcal{O}_k = \mathcal{B}^4_1\mathcal{B}^4_2\mathcal{B}^4_3\mathcal{B}^4_4\mathcal{B}^4_5$, where $\mathcal{B}_i$ are prime ideals of $k$. Since $[\mathcal{B}_i]$ and $[\mathcal{B}_j]$ with $i\,\neq\,j\in\{1,2,3,4\}$ are in $C_{k,5}$ we have $\langle[\mathcal{B}_i],[\mathcal{B}_j]\rangle \, \subset\, C_{k,5}$. To prove that $\langle[\mathcal{B}_i],[\mathcal{B}_j]\rangle  =  C_{k,5}$, we should prove that $\langle[\mathcal{B}_i],[\mathcal{B}_j]\rangle$ has order $25$. It is sufficient to prove that $[\mathcal{B}_i^{a_1}\mathcal{B}_j^{a_2}] = 1$ for $a_1,a_2\in\{0,1,2,3,4,\}$ if and only if $a_1 = a_2 = 0$.
\begin{center}
$[\mathcal{B}_i^{a_1}\mathcal{B}_j^{a_2}] = 1\Longrightarrow\, \exists\, \beta \in \mathcal{O}_{k}\, |\,\mathcal{B}_i^{a_1}\mathcal{B}_j^{a_2} = \beta\mathcal{O}_{k}$\\
\hspace{3.4cm}$\Longrightarrow\, \mathcal{N}_{k/k_0}(\mathcal{B}_i^{a_1}\mathcal{B}_j^{a_2}) = \mathcal{N}_{k/k_0}(\beta\mathcal{O}_{k})$\\
\hspace{5.7cm}$\Longrightarrow\, \lambda^{a_1+a_2}\mathcal{O}_{k_0} = \mathcal{N}_{k/k_0}(\beta)\mathcal{O}_{k_0}$ with $\lambda = 1-\zeta_5$\\
\hspace{3.3cm} $\Longrightarrow\, \exists\, \epsilon \in E_{k_0}\, |\,\lambda^{a_1+a_2} = \epsilon\mathcal{N}_{k/k_0}(\beta)$ 
\end{center}

According to [\cite{Mani}, theorem 5.15] $E_{k_0}\subset \mathcal{N}_{k/k_0}(k^{*})$ because  $q^{*} = 2$ 
\begin{center}
\hspace{3.3cm}$\Longrightarrow\, \exists \alpha \in k-\{0\}\, |\,\lambda^{a_1+a_2} = \mathcal{N}_{k/k_0}(\alpha)$ 
\end{center}
that is to say $\lambda^{a_1+a_2}$ is norm in $k = k_0(\sqrt[5]{p})$, where $p = \pi_1\pi_2$ in $k_0$. Hence we have:
\begin{center}
\large{$\left(\frac{\lambda^{a_1+a_2},\pi_1\pi_2}{\mathcal{P}}\right)_5 = 1$}
\end{center}
for all ideals $\mathcal{P}$ of $k_0$.\\
In particular, we calculate this symbol for $\mathcal{P} = \pi_1\mathcal{O}_{k_0}$ or $\mathcal{P} = \pi_2\mathcal{O}_{k_0}$. Using $(1)$ of properties \ref{normprop} xe have,
\begin{center}
$A = $\large{$\left(\frac{\lambda^{a_1+a_2},\pi_1\pi_2}{\pi_1\mathcal{O}_{k_0}}\right)_5 = \left(\frac{\lambda^{a_1},\pi_1}{\pi_1}\right)_5.\left(\frac{\lambda^{a_1},\pi_2}{\pi_1}\right)_5.\left(\frac{\lambda^{a_2},\pi_1}{\pi_1}\right)_5.\left(\frac{\lambda^{a_2},\pi_2}{\pi_1}\right)_5$}
\end{center}
then, we get:
\begin{itemize}
\item[-] \large{$\left(\frac{\lambda^{a_1},\pi_1}{\pi_1}\right)_5 = \left(\frac{\lambda,\pi_1}{\pi_1}\right)_5^{a_1} = \left(\frac{\lambda}{\pi_1}\right)_5^{a_1}$}, by $(1)$ and $(9)$ of properties \ref{normprop}.

\item[-] \large{$\left(\frac{\lambda^{a_1},\pi_2}{\pi_1}\right)_5 = \left(\frac{\lambda,\pi_2}{\pi_1}\right)_5^{a_1} = \left(\frac{\lambda}{\pi_1}\right)_5^{0\times a_1} = 1$}, by $(3)$ and $(4)$ of properties \ref{normprop}.

\item[-] \large{$\left(\frac{\lambda^{a_2},\pi_1}{\pi_1}\right)_5 = \left(\frac{\lambda,\pi_1}{\pi_1}\right)_5^{a_2} = \left(\frac{\lambda}{\pi_1}\right)_5^{a_2}$}, by $(1)$ and $(9)$ of properties \ref{normprop}.

\item[-] \large{$\left(\frac{\lambda^{a_2},\pi_2}{\pi_1}\right)_5 = \left(\frac{\lambda,\pi_2}{\pi_1}\right)_5^{a_2} = \left(\frac{\lambda}{\pi_1}\right)_5^{0\times a_2} = 1$}, by $(3)$ and $(4)$ of properties \ref{normprop}.

\end{itemize}
we get that, \large{$A = \left(\frac{\lambda}{\pi_1}\right)_5^{a_1+a_2}=\,1$}. Since $\pi_1$ and $\pi_2$ play symmteric roles, then $B$ = \large{$\left(\frac{\lambda^{a_1+a_2},\pi_1\pi_2}{\pi_2\mathcal{O}_{k_0}}\right)_5 = \left(\frac{\lambda}{\pi_2}\right)_5^{a_1+a_2}$}, 
and since \large{$\left(\frac{\lambda^{a_1+a_2},\pi_1\pi_2}{\mathcal{P}}\right)_5 = 1$}, for all prime ideals of $k_0$, then $A = B = 1$, namely \large{$\left(\frac{\lambda}{\pi_1}\right)_5^{a_1+a_2} = \left(\frac{\lambda}{\pi_2}\right)_5^{a_1+a_2} = 1$}. Since $5$ is not a quintic residue modulo $p$, we have $a_1+a_2 = 0$ or $5$, then $a_1 = a_2 = 0$, because otherwise $a_1\,+\,a_2 = 5$, so we need to distinguish two cases:\\

- If $a_1 = 1$ and $a_2 = 4$, then $[\mathcal{B}_i\mathcal{B}_j^4] = 1$ means that $[\mathcal{B}_i] = [\mathcal{B}_j]^{-4}$ and since $[\mathcal{B}_j]^5 = 1$ because $[\mathcal{B}_i]$ is $5$-class, we get $[\mathcal{B}_i] = [\mathcal{B}_j]$ which is impossible.\\

- If $a_1 = 2$ and $a_2 = 3$, then $[\mathcal{B}_i^2\mathcal{B}_j^3] = 1$ means that $[\mathcal{B}_i]^2 = [\mathcal{B}_j]^{-3}$, so $[\mathcal{B}_i] = [\mathcal{B}_j]$ which is impossible.\\
Thus the case $a_1\,+\,a_2 = 5$ cannot occur. Hence $\langle[\mathcal{B}_i],[\mathcal{B}_j]\rangle$ has order $25$, and since $\langle[\mathcal{B}_i],[\mathcal{B}_j]\rangle\,\subset C_{k,5}$ we deduce that $C_{k,5} = \langle[\mathcal{B}_i],[\mathcal{B}_j]\rangle$. 
\begin{corollary}
Using the same notation as above, we have:
\item[$$(1)$$] $C_{k,5}^{(\sigma)} = C_{k,5}^+ = \langle[\mathcal{B}_1\mathcal{B}_2\mathcal{B}_3\mathcal{B}_4\mathcal{B}_5]\rangle$.
\item[$$(2)$$] $C_{k,5}^- = \langle[\mathcal{B}_i]^{1-\tau^2}\rangle$ for $(i=1,2,3,4,5)$.
\item[$(3)$] The $5$-class group can be generated also by: 
\begin{center}
$C_{k,5} = \langle[\mathcal{B}_1\mathcal{B}_2\mathcal{B}_3\mathcal{B}_4\mathcal{B}_5], [\mathcal{B}_i]^{1-\tau^2}\rangle$ for $(i=1,2,3,4,5)$.
\end{center}
\end{corollary}

\begin{proof}
\item[$(1)$] The fact that the ideals $\mathcal{B}_i$ are not principal by the same reasoning as above, we prove that $[\mathcal{B}_1\mathcal{B}_2\mathcal{B}_3\mathcal{B}_4\mathcal{B}_5]^\sigma = [\mathcal{B}_1\mathcal{B}_2\mathcal{B}_3\mathcal{B}_4\mathcal{B}_5]$ and  $[\mathcal{B}_1\mathcal{B}_2\mathcal{B}_3\mathcal{B}_4\mathcal{B}_5]^{\tau^2} = [\mathcal{B}_1\mathcal{B}_2\mathcal{B}_3\mathcal{B}_4\mathcal{B}_5]$ by applying the decomposition of $5$ in the normal closure $k$.
\item[$(2)$] Since $\mathcal{B}_i$ are not principal then $[\mathcal{B}_i]^{1-\tau^2} \in C_{k,5}^-$
\item[$(3)$] That comes from $C_{k,5} \cong C_{k,5}^+ \times C_{k,5}^-$
\end{proof}
\section{Numerical examples}
Using the system PARI/GP [\cite{PRI}], we illustrate our main result Theorem \ref{thp} by numerical examples.
\subsection{Case 1: $n = 5^{e_1}p^{e_2}\,\not\equiv\,\pm1,\pm7\,(\mathrm{mod}\,25)$ with $p\,\not\equiv\,-1\,(\mathrm{mod}\,25)$}
In this case we have: $C_{k,5} = \langle[\mathcal{P}],[\mathcal{L}]^{1-\tau^2}\rangle$, where $\mathcal{P}$ is a prime ideal of $k$ above $p$, and $\mathcal{L}$ is prime ideals of $k$ above $l$ such that $l\,\neq\,p$. The following table verifies, for some primes $p\,\not\equiv\,-1\,(\mathrm{mod}\,25)$ and $l\,\neq\,p$, such that $5$ and $l$ are not quintic residues modulo $p$ and $C_{k,5}$ is of type $(5,5)$, that the ideals $\mathcal{P}$ and $\mathcal{L}$ are not principal, and are of order $5$
\begin{center}
 Table 1: 
\end{center}

\begin{tabular}{|c|c|c|c|c|c|c|}
\hline 
$p$ & $l$ & Type of $C_{k,5}$ & Is principal $\mathcal{P}$ & Is principal $\mathcal{L}$ & Is principal $\mathcal{P}^5$ & Is principal $\mathcal{L}^5$  \\ 
\hline 
19 & 2 & (5,5) & [4,0] & [1,0] & [0,0] & [0,0] \\ 
19 & 113 & (5,5) & [4,0] & [4,0] & [0,0] & [0,0] \\ 
29 & 43 & (5,5) & [6,0,0,0] & [6,0,1,0] & [0,0,0,0] & [0,0,0,0] \\ 
29 & 149 & (5,5) & [6,0,0,0] & [2,0,0,0] & [0,0,0,0] & [0,0,0,0] \\ 
59 & 67 & (5,5) & [4,0,0,0] & [6,0,0,0] & [0,0,0,0] & [0,0,0,0] \\ 
79 & 97 & (5,5) & [21,21,0,0] & [28,28,0,0] & [0,0,0,0] & [0,0,0,0] \\ 
79 & 307 & (5,5) & [21,21,0,0] & [7,7,0,0] & [0,0,0,0] & [0,0,0,0] \\ 
89 & 19 & (5,5) & [1,0] & [4,0] & [0,0] & [0,0] \\ 
89 & 101 & (5,5) & [1,0] & [4,4] & [0,0] & [0,0] \\ 
109 & 13 & (5,5) & [2,0] & [1,0] & [0,0] & [0,0] \\ 
109 & 103 & (5,5) & [2,0] & [1,0] & [0,0] & [0,0] \\ 
179 & 157 & (5,5) & [3,0] & [2,0] & [0,0] & [0,0] \\ 
299 & 83 & (5,5) & [8,12,0,0] & [12,8,0,0] & [0,0] & [0,0] \\ 
239 & 3 & (5,5) & [3,0] & [2,0] & [0,0] & [0,0] \\ 
269 & 157 & (5,5) & [1,0] & [3,0] & [0,0] & [0,0] \\

\hline 
\end{tabular} 

\newpage
\subsection{Case 2: $n = p^{e_1}q^{e_2}\equiv\pm1,\pm7\,(\mathrm{mod}\,25)$ with $p\,\not\equiv\,-1\,(\mathrm{mod}\,25)$ and $q\,\not\equiv\,\pm7\,(\mathrm{mod}\,25)$}
In this case we have: $C_{k,5} = \langle[\mathcal{P}],[\mathcal{L}]^{1-\tau^2}\rangle$, where $\mathcal{P}$, $\mathcal{L}$ as above.
\begin{center}
Table 2 
\end{center}
\begin{tabular}{|c|c|c|c|c|c|c|c|}
\hline 
$p$ & $q$&$l$&Type of $C_{k,5}$ & Is principal $\mathcal{P}_1$ & Is principal $\mathcal{P}_2$ & Is principal $\mathcal{P}_1^5$ & Is principal $\mathcal{P}_2^5$  \\ 
\hline 
19 & 3& 53&(5,5) & [1,0] & [3,0] & [0,0] & [0,0] \\ 
19 & 3& 67&(5,5) & [1,0] & [4,0] & [0,0] & [0,0] \\ 
19 & 53& 7&(5,5) & [3,0] & [1,0] & [0,0] & [0,0] \\ 
19 & 53& 23&(5,5) & [3,0] & [1,0] & [0,0] & [0,0] \\ 
29 & 17& 157&(5,5) & [28,14,0,0] & [14,7,0,0] & [0,0,0,0] & [0,0,0,0] \\ 
59 & 2& 13&(5,5) & [0,2] & [0,2] & [0,0] & [0,0] \\ 
59 & 2& 47&(5,5) & [0,2] & [0,3] & [0,0] & [0,0] \\ 
59 & 23& 127&(5,5) & [2,0,0,0] & [6,0,0,0] & [0,0,0,0] & [0,0,0,0] \\ 
89 & 37& 61&(5,5) & [3,0] & [3,0] & [0,0] & [0,0] \\ 
89 & 37& 73&(5,5) & [3,0] & [2,0] & [0,0] & [0,0] \\ 
109 & 23& 47&(5,5) & [28,0,0,0] & [14,0,0,0] & [0,0,0,0] & [0,0,0,0] \\ 
359 & 2& 61&(5,5) & [2,1] & [3,1] & [0,0] & [0,0] \\ 
359 & 2& 97&(5,5) & [2,1] & [4,2] & [0,0] & [0,0] \\ 
409 & 2& 157&(5,5) & [2,0] & [1,0] & [0,0] & [0,0] \\ 
409 & 2& 163&(5,5) & [2,0] & [3,0] & [0,0] & [0,0] \\ 
509 & 2& 53&(5,5) & [2,0] & [3,0] & [0,0] & [0,0] \\ 
509 & 2& 79&(5,5) & [2,0] & [4,0] & [0,0] & [0,0] \\ 

\hline 
\end{tabular} 
\vspace{1cm}
\subsection{Case 3: $n = p^e\equiv\pm1,\pm7\,(\mathrm{mod}\,25)$ with $p\equiv-1\,(\mathrm{mod}\,25)$}
In this case we have: $C_{k,5} = \langle[\mathcal{B}_1],[\mathcal{B}_2]\rangle$, where $\mathcal{B}_i$ are prime ideals of $k$ above $5$. The following table verifies that the ideals $\mathcal{B}_1$ and $\mathcal{B}_2$ are not principal and are of order $5$ such that \large{$\left(\frac{5}{p}\right)_5\,\neq\,1$}
\begin{center}
 Table 3: 
\end{center}

\begin{tabular}{|c|c|c|c|c|c|}
\hline 
$p$ & Type of $C_{k,5}$ & Is principal $\mathcal{B}_1$ & Is principal $\mathcal{B}_2$ & Is principal $\mathcal{B}_1^5$ & Is principal $\mathcal{B}_2^5$  \\ 
\hline 
149 & (5,5) & [1,0] & [1,0] & [0,0] & [0,0] \\ 
199 & (5,5) & [6,0,0,0] & [6,0,0,0] & [0,0,0,0] & [0,0,0,0] \\ 
499 & (5,5) & [1,4] & [1,4] & [0,0] & [0,0] \\ 
599 & (5,5) & [3,0] & [3,0] & [0,0] & [0,0] \\ 
2099 & (5,5) & [1,0] & [1,0] & [0,0] & [0,0] \\ 
2549 & (5,5) & [1,0] & [1,0] & [0,0] & [0,0] \\ 
2699 & (5,5) & [1,0] & [1,0] & [0,0] & [0,0] \\ 
\end{tabular}
\newpage
\begin{tabular}{|c|c|c|c|c|c|}
\textcolor{white}{$p$} & \textcolor{white}{Type of $C_{k,5}$} & \textcolor{white}{Is principal $\mathcal{B}_1$} & \textcolor{white}{Is principal $\mathcal{B}_2$} & \textcolor{white}{Is principal $\mathcal{B}_1^5$} & \textcolor{white}{Is principal $\mathcal{B}_2^5$}  \\
2749 & (5,5) & [1,0] & [1,0] & [0,0] & [0,0] \\ 
3299 & (5,5) & [1,0] & [1,0] & [0,0] & [0,0] \\ 
4049 & (5,5) & [1,0] & [1,0] & [0,0] & [0,0] \\ 
4099 & (5,5) & [8,0,0,0] & [8,0,0,0] & [0,0,0,0] & [0,0,0,0] \\ 
4349 & (5,5) & [1,0] & [1,0] & [0,0] & [0,0] \\ 
4549 & (5,5) & [1,0] & [1,0] & [0,0] & [0,0] \\ 
4999 & (5,5) & [1,0] & [1,0] & [0,0] & [0,0] \\ 
5099 & (5,5) & [8,0,0,0] & [8,0,0,0] & [0,0,0,0] & [0,0,0,0] \\ 
5399 & (5,5) & [1,0] & [1,0] & [0,0] & [0,0] \\ 
5749 & (5,5) & [1,0] & [1,0] & [0,0] & [0,0] \\ 
6199 & (5,5) & [0,4,0,0,0,0,0,0] & [0,4,0,0,0,0,0,0] & [0,0,0,0,0,0,0,0] & [0,0,0,0,0,0,0,0] \\ 
6299 & (5,5) & [1,0] & [1,0] & [0,0] & [0,0] \\ 
6599 & (5,5) & [1,0] & [1,0] & [0,0] & [0,0] \\ 
\hline 
\end{tabular}


\begin{thebibliography}{XX}
\bibitem{Bahmanpour}
K. Bahmanpour,  \emph{Prime numbers $p$ with expression $p = a^2\pm ab \pm b^2$}, Journal of Number Theory 166 (2016) 208-218.\label{Bahmanpour}


\bibitem{conductor}
W.E.H. Berwick, \emph{Integral bases}, Cambridge Tracts in Math. and Math.Phys, Vol 22, 1927.\label{conductor}

\bibitem{Dedkind}
R.Dedekind.  \emph{Ueber die Anzahl der Idealklassen in reinen kubischen Zahlk\"orpern}, J f\"ur reine und angewandte mathematik, Bd 121 (1900), 40-123.\label{Dedkind}

\bibitem{FOU}\label{FOU}
F.Elmouhib, M.Talbi, and A.Azizi, \emph{5-rank of ambiguous class groups of quintic Kummer extensions}, Proc Math Sci 132, 12 (2022). \url{https://doi.org/10.1007/s12044-022-00660-z}

\bibitem{Euc2} D. Grant,  \emph{A proof of quintic reciprocity using the arithmetic of $y^2 = x^5+\frac{1}{4}$}. ACTA ARITHMETICA (1996).
%
\bibitem{Euc2} G.Gras,  \emph{Sur les $l$-classes d'id\'{e}ux dans les extensions cycliques relatives de degr\'{e} premier impaire $l$}.  Annales de l'institut Fourier, (1973).

\bibitem{Hass} 
H. Hasse, \emph{Neue Begr\"ndung und Verallgemeinerung der Theorie des Normenrest Symbols}, Journal f\"ur reine und ang. Math. 162 (1930), 134-143.\label{Hass}

\bibitem{Hec} E.Hecke,  \emph{Lectures on the Theory of Algebraic Numbers}, GTM , Vol. 77, Springer-Verlag 1981.\label{Hec}

\bibitem{Irland} K.Ireland and M.Rosen,  \emph{A Classical Introduction to modern Number Theory}. Graduate Texts in Mathematics 84, Springer-Verlag (1982).\label{Irland}

\bibitem{Jan} G.J Janus,  \emph{Algebraic Number Fields}. Academic Press, New York-London (1973).\label{Jan}

\bibitem{Mani}
M.Kulkarni, D. Majumdar and B.Sury  \emph{$l$-class groups of cyclic extension of prime degree $l$}, J. Ramanujan Math. Soc. 30, No.4 (2015), 413-454.\label{Mani}
%
\bibitem{Kobayashi}
 H. Kobayashi,  \emph{Class numbers of pure quintic fields}, Journal of Number Theory 160 (2016) 463-477.\label{Kobayashi}

\bibitem{Mayer}
D.C. Mayer, \emph{Differential Principal Factors and Polya Property of Pure Metacyclic Fields}, International Journal of Number Theory, Vol. 15, No. 10 (2019), pp. 1983-2025.\label{Mayer}

\bibitem{Pa}
C. Parry,  \emph{Class number relations in pure quintic felds}, Symposia Mathematica. 15 (1975), 475-485.\label{Pa}
%

\bibitem{Euc2} P. Samuel.  \emph{Th\'{e}orie alg\'{e}brique des nombres}. Hermann, editeurs des sciences et des arts (1971).
%
\bibitem{cyc}
L.C. Washington,  \emph{Introduction to Cyclotomic Fields}, Springer-Verlag New York Inc (1982).\label{cyc}

\bibitem{PRI}
The PARI Group, PARI/GP, Version 2.4.9, Bordeaux, 2017, http://pari.math.u-bordeaux.fr.\label{PRI}
\end{thebibliography}
\end{document}